\documentclass{amsart}
\usepackage{amsmath}
\usepackage{amsthm}
\usepackage{amssymb}

\theoremstyle{plain}

\theoremstyle{change}
\newtheorem{thm}{Theorem}[section]
\newtheorem{prop}[thm]{Proposition}

\theoremstyle{remark}

\theoremstyle{definition}
\newtheorem{defn}[thm]{Definition}
\numberwithin{equation}{section}

\renewcommand{\labelenumi}{(\theenumi)}

\def\Vec#1{\mbox{\boldmath $#1$}}

\usepackage{amscd}

\title{On self-dual simple types of $p$-adic classical groups}
\author{Kazutoshi Kariyama and Michitaka Miyauchi}

\begin{document}
\maketitle

\begin{abstract}
Let $G$ be a classical group over a non-Archimedean local field of odd residual characteristic. Using recent work of S. Stevens, we define a certain kind of semisimple stratum, called {\it good}, and show that it provides a simple type in $G$ which is an analogue of the simple type for $GL(N,F)$ defined by Bushnell and Kutzko. Furthermore, we define a {\it self-dual} simple type in $G$.
\end{abstract}

\vspace{2mm}
{\footnotesize

 }
 
\vspace{2mm}

\begin{center}
I{\scriptsize NTRODUCTION}
\end{center}

In order to classify the smooth representations of the general linear group $G = GL(N,F)$ of a non-Archimedean local field $F$, a family of representations of certain compact open subgroups of $G$, which are called {\it simple types}, were constructed by Busnell and Kutzko \cite{BK1, BK3}. 
Each simple type is associated with the inertial class $[M,\sigma]_G$ of a supercuspidal representation $\sigma$ of a Levi subgroup $M$ in such a way that $M \simeq GL(m,F) \times GL(m,F) \times \cdots \times GL(m,F)$ and $\sigma \simeq \sigma_0 \times \sigma_0 \times \cdots \times \sigma_0$ for an irreducible supercuspidal representation $\sigma_0$ of $GL(m,F)$ with $N = md$.

For a symplectic group $Sp(2N,F)$ over a non-Archimedean local field $F$ of odd residual characteristic, a type is constructed in Blondel \cite{Bl1}. It is associated with the inertial class $[M,\pi_0 \otimes \pi_0]$, which consists of a maximal Levi subgroup $M \simeq GL(N,F) \times GL(N,F)$ and the tensor product $\pi_0 \otimes \pi_0$ of a self-contragradient irreducible supercuspidal representation $\pi_0$ of $GL(N,F)$. 
These results are generalized to a maximal Levi subgroup of the other classical groups in Goldberg, Kutzko and Stevens \cite{GKS}, and to a Levi subgroup, which is not always maximal, of $Sp(2N,F)$ in Blondel \cite{Bl2}.  By the methods of Bushnell and Kutzko \cite{BK1}, in Kariyama \cite{Ka}, they can also be generalized to a Levi subgroup of an unramified unitary group over a non-Archimedean local field $F_0$ of odd residual characteristic. Furthermore, in \cite{St2}, many types in the general classical groups over $F_0$ are exhibited.

The purpose of this article is to define a {\it good} skew semisimple stratum and construct a {\it simple type} attached to this stratum for the general classical groups $G$ over $F_0$ by using the results of \cite{St2}. This simple type plays the same role as that for $GL(N,F)$. Furthermore, following a result in \cite{Ka}, a {\it self-dual} simple type in $G$ can be defined. This generalizes those types constructed in \cite{Bl1, Bl2, GKS, Ka}.

Let $(V,h)$ be the non-degenerate Hermitian form that defines the group $G$. We take a skew semisimple stratum $[\Lambda,n,0,\beta]$, which consists of a lattice sequence $\Lambda$ in $V$, integers $n > 0$, and a semisimple skew element $\beta = \sum_{i=1}^{\ell+1}\beta_i$ in $\mathrm{End}_F(V)$. The $F$-algebra $F[\beta]$ generated by $\beta$ decomposes $V$ into a direct sum of simple modules $V = \bigoplus_{i=1}^{\ell+1}V^i$. We say that $[\Lambda,n,0,\beta]$ is {\it good}, if $V$ has another decomposition $V = \bigoplus_{j=-m}^mW^{(j)}$ which is {\it self-dual}, {\it exactly subordinate} to $[\Lambda,n,0,\beta]$ in the sense of \cite{St2} and has the property that for $j \ne 0$, $W^{(j)}$ is contained in $V^{\ell+1}$ and has constant $F$-dimension (see Definition 2.1).

The latter decomposition of $V$ yields a parabolic $P = MU$ of $G$ with $M \simeq G^{(0)} \times GL(N/m,F)^m$, for some positive integers $m, N$ with $m \vert N$, where $G^{(0)}$ is a subgroup of the same type as $G$. From the lattice $\Lambda$, we obtain compact open subgroups $J^1_P \subset J_P$ of $G$ such that $J_P/J_P^1 \simeq \overline{G}^{(0)} \times GL(f,k_{E'})^m$, for some positive integer $f$ and a finite field $k_{E'}$, and a certain irreducible representation $\kappa_P$ of $J_P$, as in \cite{St2}.

For an irreducible representation $\tau$ of $J_P$ that is trivial on $J_P^1$, we say that the representation $\lambda_P = \kappa_P \otimes \tau$ of $J_P$ is a {\it simple type} if $\tau$ induces a representation of $J_P/J_P^1$ that contains a certain irreducible cuspidal representation of the identity component of $J_P/J_P^1$ as a finite reductive group (Definition 4.5). If $(J_P,\lambda_P)$ is a simple type, 
there exist irreducible supercuspidal representations $\pi_{\mathrm{cusp}}, \widetilde{\pi}^{(1)}, \cdots, \widetilde{\pi}^{(m)}$ of $G^{(0)},\ GL(N/m,F)$, respectively, such that
$(J_P,\lambda_P)$ is a $[M,\pi_M]_G$-type in $G$ in the sense of Bushnell-Kutzko \cite{BK2}, where $\pi_M = \pi_{\rm{cusp}} \otimes \bigotimes_{j=1}^m \widetilde{\pi}^{(j)}$ (Theorem 6.3).

There exist certain Weyl group elements $s_j$, for $1 \le j \le m$, and the conjugation on $M$ by $s_j$ induces an involution $\sigma_j$ on the $j$-th factor $GL(N/m,F)$ of $M$. It leaves the $j$-th factor $GL(f,k_{E'})$ of $J_P/J_P^1$ stable. A simple type $(J_P,\lambda_P)$, with $\lambda_P = \kappa_P \otimes \tau$, is called {\it self-dual} if the $j$-th factor $\overline{\widetilde{\tau}}^{(j)}$ on $GL(f,k_{E'})$ of the representation $\tau$ satisfies $\overline{\widetilde{\tau}}^{(j)} \circ \sigma_j \simeq \overline{\widetilde{\tau}}^{(j)}$, for $1 \le j \le m$ (this is equivalent to Definition 5.2). 
We show that if $\pi$ is an irreducible smooth representation of $G$ that contains a self-dual simple type $\lambda_P$, then there exists an irreducible self-dual supercuspidal representation $\rho$ of $GL(N/m,F)$ and a representation $\pi_{\rm{cusp}}$ of $G^{(0)}$ as above such that $\pi$ is equivalent to a $G$-subquotient of the parabolically induced representation $\mathrm{Ind}_P^G (\pi_{\rm{cusp}} \otimes \rho \nu^{x_1} \otimes \cdots \otimes \rho \nu^{x_m})$, where $\nu = \vert \det \vert_F$ on $GL(N/m,F)$ and $x_1,\cdots,x_m \in \mathbb{C}$, (Theorem 6.2).

The remainder of this paper is structured as follows: In Section 1, we introduce notation and provide some required definitions. In Section 2, we define a good skew semisimple stratum. 
In Section 3, we recall the notion of a $\beta$-extension, which is defined in \cite{St2}. We define a simple type $(J_P,\lambda_P)$ and a self-dual simple type in Sections 4 and 5. 
In Section 6, we prove that a simple type $(J_P,\lambda_P)$ is a type in the sense of \cite{BK2}. In fact it is a $G$-cover, as explained above.

 {\it Notation}: We denote by $\mathbb{N}$, $\mathbb{Z}$, and $\mathbb{C}$ the set of natural numbers, the ring of rational integers, and the field of complex numbers, respectively. For a ring $R$, we denote the multiplicative group of $R$ by $R^\times$.

\section{Preliminaries}     

We recall the notation used in Bushnell-Kutzko \cite{BK1, BK3}, Stevens \cite{St1, St2}.

Let $F$ be a non-Archimedean local field with Galois involution $^-$, and with fixed field $F_0 = \{x \in F \vert\ \overline{x} = x\}$. We allow $F_0$ to equal $F$. Let $\mathfrak{o}_F$ be the ring of integers of $F$, $\mathfrak{p}_F$ the maximal ideal of $\mathfrak{o}_F$, $\varpi_F$ a uniformizer of $F$, and $k_F = \mathfrak{o}_F/\mathfrak{p}_F$ the residue class field. We will assume throughout this study that the residual characteristic $p$ of $F$ is not equal to $2$.

Let $V$ be a finite-dimensional vector space over $F$ equipped with a non-degenerate $\varepsilon$-hermitian form $h$, where $\varepsilon \in \{+1,-1\}$. Put $A = \mathrm{End}_F(V)$, and let $a \mapsto \overline{a}$ be the (anti-)involution on $A$ defined by the form $h$:
\begin{equation*}
h(av,w) = h(v,\overline{a}w)\ (a \in A,\ v, w \in V).
\end{equation*}
Put $\widetilde{G} = A^\times = \mathrm{Aut}_F(V)$, and define an automorphism $\sigma$ of order 2 on $\widetilde{G}$ by $\sigma(g) = \overline{g}^{-1}\ (g \in \widetilde{G})$.
Put $\Sigma = \{1,\sigma\}$, where $1$ denotes the identity on $A$, and set
\begin{equation*}
G^+ = \widetilde{G}^\Sigma = \{g \in \widetilde{G} \vert\ \sigma(g) = g\}.
\end{equation*}
Denote by $G$ the identity component of $G^+$. Then $G$ is either a unitary group, a symplectic group, or a special orthogonal group over $F_0$. For a subgroup $\widetilde{H}$ of $\widetilde{G}$ with $\sigma(\widetilde{H}) = \widetilde{H}$, write $H^+ = \widetilde{H} \cap G^+,\ H = \widetilde{H} \cap G.$

An $\mathfrak{o}_F$-{\it lattice sequence} in $V$ is a function $\Lambda: \mathbb{Z} \to \{\mathfrak{o}_F$-lattices in $V \}$ that satisfies
\begin{enumerate}
  \item $n \ge m$ implies $\Lambda(n) \subset \Lambda(m)$, \
  \item there exists a positive integer $e = e(\Lambda \vert \mathfrak{o}_F)$ (the $\mathfrak{o}_F$-{\it period}) such that $\Lambda(n+e) = \mathfrak{p}_F\Lambda(n)$,\hspace{2mm} $(n \in \mathbb{Z})$.
\end{enumerate}

We say that an $\mathfrak{o}_F$-lattice sequence $\Lambda$ is {\it strict}, if $\Lambda(n) \supsetneq \Lambda(n+1)$\ $(n \in \mathbb{Z})$.

For an $\mathfrak{o}_F$-lattice $L$ in $V$, we define the {\it dual lattice} $L^\#$ by $L^\# = \{v \in V \vert\ h(v,L) \subset \mathfrak{p}_F\}$.
An $\mathfrak{o}_F$-lattice sequence $\Lambda$ in $V$ is called {\it self-dual}, if there exists an integer $d$ such that $\Lambda(k)^\# = \Lambda(d-k)\ (k \in \mathbb{Z})$.

From an $\mathfrak{o}_F$-lattice sequence $\Lambda$ in $V$, we obtain a filtration on $A$ by
\begin{equation*}
\mathfrak{a}_n = \mathfrak{a}_n(\Lambda) = \{x \in A \vert\ x\Lambda(k) \subset \Lambda(k+n)\ (k \in \mathbb{Z})\}\ (n \in \mathbb{Z}).
\end{equation*}
In particular, $\mathfrak{a}_0$ is a hereditary $\mathfrak{o}_F$-order in $A$ and $\mathfrak{a}_1$ is its Jacobson radical. This filtration defines a valuation $\nu_\Lambda$ on $A$ by $\nu_{\Lambda}(x) = \sup \{n \in \mathbb{Z} \vert\ x \in \mathfrak{a}_n\},\ (x \in A)$, with $\nu_\Lambda(0) = + \infty$.

From an $\mathfrak{o}_F$-lattice sequence $\Lambda$ in $V$, we obtain open compact subgroups of $\widetilde{G}$ by
\begin{center}
$\widetilde{P} = \widetilde{P}(\Lambda) = \mathfrak{a}_0(\Lambda)^\times,$\\
$\widetilde{P}_n = \widetilde{P}_n(\Lambda) = 1 + \mathfrak{a}_n(\Lambda)\ (n > 0).$
\end{center}
Then the $\widetilde{P}_n$ $(n > 0)$ are normal subgroups of $\widetilde{P}$ and form a filtration of $\widetilde{G}$. If $\Lambda$ is self-dual, then we obtain open compact subgroups of $G^+$ and $G$ from these groups by
\begin{center}
$P^+ = P^+(\Lambda) = \widetilde{P}(\Lambda) \cap G^+,\ P = P(\Lambda) = \widetilde{P}(\Lambda) \cap G$,\\
$P_n = P_n(\Lambda) = \widetilde{P}_n(\Lambda) \cap G\ (n > 0)$.
\end{center}
The quotient $\mathcal{G} = P/P_1$ is the group of $k_{F_0}$-rational points of a reductive algebraic group defined over $k_{F_0}$, where $k_{F_0}$ denotes the residue class field of $F_0$. We note that it is not always connected. 
We denote by $P^\mathrm{o} = P^\mathrm{o}(\Lambda)$ the inverse image of the identity component $\mathcal{G}^\mathrm{o}$ of $\mathcal{G} = P/P_1$ in $P = P(\Lambda)$. Then we have $\mathcal{G}^\mathrm{o} = P^\mathrm{o}/P_1$.

A {\it stratum} in $A$ is a 4-tuple $[\Lambda,n,r,b]$, where $\Lambda$ is an $\mathfrak{o}_F$-lattice sequence in $V$, $n \in \mathbb{Z}$, $r \in \mathbb{Z}$ with $n \ge r \ge 0$, and $b \in \mathfrak{a}_{-n}(\Lambda)$. A stratum $[\Lambda,n,r,\beta]$ is called {\it simple}, if it satisfies the following conditions:
\begin{enumerate}
  \item the algebra $E = F[\beta]$ is a field;\
  \item $\Lambda$ is an $\mathfrak{o}_E$-lattice sequence (which we denote by $\Lambda_{\mathfrak{o}_E}$);\
  \item $\nu_\Lambda(\beta) = -n$;\
  \item $k_0(\beta,\Lambda) < -r$,
\end{enumerate}
where $k_0(\beta,\Lambda)$ is the integer defined in \cite[\S 5]{St1} (cf. \cite[(1.5)]{BK1}).

A stratum $[\Lambda,n,r,b]$ is called {\it null}, if $n = r$ and $b = 0$.

Let $[\Lambda,n,r,\beta]$ be a stratum in $A$, and $V = \bigoplus_{i=1}^{\ell} V^i$ a direct sum decomposition of $V$ into $F$-subspaces. We say that the $F$-decomposition $V = \bigoplus_{i=1}^{\ell} V^i$ is {\it splitting} for $[\Lambda,n,r,\beta]$, if we have $\Lambda(k) = \bigoplus_{i=1}^{\ell} \Lambda^i(k)\ (k \in \mathbb{Z}),\ \beta = \sum_{i=1}^{\ell} \beta_i$, where for each $i$, $\Lambda^i(k) = \Lambda(k) \cap V^i\ (k \in \mathbb{Z})$, and for the projection $\Vec{1}^i: V \to V^i$ with kernel $\bigoplus_{j\ne i}V^j$, $\beta_i = \Vec{1}^i\beta\Vec{1}^i$.

\begin{defn} (\cite[3.2]{St1}).      
A stratum $[\Lambda,n,r,\beta]$ in $A$ is called {\it semisimple}, if either it is null or $\nu_\Lambda(\beta) = -n$ and there exists a splitting $V = \bigoplus_{i=1}^\ell V^i$ for the stratum such that
 \begin{enumerate}
 \item for $1 \le i \le \ell$, $[\Lambda^i,q_i,r,\beta_i]$ is a simple or null stratum in $A^i = \mathrm{End}_F(V^i)$, where $q_i = r$ if $\beta_i = 0,\ \nu_\Lambda(\beta_i) = -q_i$ otherwise; and \
\item for $1 \le i, j \le \ell,\ i \ne j$, the stratum $[\Lambda^i \oplus \Lambda^j,q,r,\beta_i+\beta_j]$ is not equivalent to a simple stratum or null stratum, with $q = \max \{q_i,q_j\}$.
\end{enumerate}
\end{defn}

A semisimple stratum $[\Lambda,n,r,\beta]$ is called {\it skew}, if (1) $\Lambda$ is self-dual, (2) $\overline{\beta} = - \beta$, and (3) $V = \bigoplus_{i=1}^\ell V^i$ in Definition 1.1 is orthogonal with respect to the Hermitian form $h$.

Let $[\Lambda,n,r,\beta]$ be a skew semisimple stratum in $A$. Denote by $B$ the $A$-centralizer of $\beta$, and set $\widetilde{G}_E = B^\times \cap \widetilde{G}$, $G^+_E = B^\times \cap G^+$, and $G_E = B^\times \cap G$.
Then we write
\begin{center}
$\widetilde{P}(\Lambda_{\mathfrak{o}_E}) = \widetilde{P}(\Lambda) \cap \widetilde{G}_E,\ \widetilde{P}_n(\Lambda_{\mathfrak{o}_E}) = \widetilde{P}_n(\Lambda) \cap \widetilde{G}_E\ (n \ge 1),$\\
$P(\Lambda_{\mathfrak{o}_E}) = \widetilde{P}(\Lambda_{\mathfrak{o}_E}) \cap G_E,\ P_n(\Lambda_{\mathfrak{o}_E}) = \widetilde{P}_n(\Lambda_{\mathfrak{o}_E}) \cap G_E\ (n \ge 1)$.
\end{center}
Similarly, $P^+(\Lambda_{\mathfrak{o}_E}) = \widetilde{P}(\Lambda_{\mathfrak{o}_E}) \cap G^+_E$. Denote by $P^\mathrm{o}(\Lambda_{\mathfrak{o}_E})$ the inverse image in $P(\Lambda_{\mathfrak{o}_E})$ of the identity component of the quotient $P(\Lambda_{\mathfrak{o}_E})/P_1(\Lambda_{\mathfrak{o}_E}).$

\section{Good skew semisimple strata}         

We define a skew semisimple stratum in $A$ which is called {\it good}. Assume that $[\Lambda,n,0,\beta]$ is a skew semisimple stratum in $A$ with $V = \bigoplus_{i=1}^{\ell+1} V^i$ and $\beta = \Sigma_{i=1}^{\ell+1} \beta_i$ as splitting. Then, by definition, the Hermitian form $h$ can be decomposed into an orthogonal direct sum:
 $h = \bigoplus_{i=1}^{\ell+1} h_i$ on $V = \bigoplus_{i=1}^{\ell+1} V^i,$ where each $h_i$ is the restriction of $h$ to $V^i$.

Let $V = \bigoplus_{j=-m}^m W^{(j)}$ be a decomposition of $V$ into subspaces such that
\begin{enumerate}
  \item $W^{(j)} = \bigoplus_{i=1}^{\ell+1} (W^{(j)} \cap V^i)$, for $-m \le j \le m$, and $V^i = \bigoplus_{j=-m}^m (W^{(j)} \cap V^i)$, for $1 \le i \le \ell+1$,\
  \item $W^{(j)} \cap V^i$ is an $E_i$-subspace of $V^i$, for $-m \le j \le m$ and $1 \le i \le \ell+1$.
\end{enumerate}
By \cite[5.1]{St2}, we say that $V = \bigoplus_{j = -m}^m W^{(j)}$ is {\it properly subordinate to } $[\Lambda,n,0,\beta]$, if the following conditions are satisfied: 
\begin{enumerate}
  \item $\Lambda(n) = \bigoplus_{j = -m}^m \Lambda^{(j)}(n)\ (n \in \mathbb{Z})$, where $\Lambda^{(j)}(n) = \Lambda(n) \cap W^{(j)}$,
  \item For any integers $k$ and $i$, with $1 \le i \le \ell+1$, there exists at most one $j$, $-m \le j \le m$, such that
  \begin{equation*}
  (\Lambda(k) \cap W^{(j)} \cap V^i) \supsetneq (\Lambda(k+1) \cap W^{(j)} \cap V^i).
  \end{equation*}
\end{enumerate}

Moreover, by \cite[5.3]{St2}, $V = \bigoplus_{j = -m}^m W^{(j)}$ is called {\it self-dual}, if the orthogonal complement $(W^{(j)})^\bot$ of $W^{(j)}$ is equal to $\bigoplus_{k \ne j} W^{(k)}$, with respect to the form $h$.

\begin{defn}    
Let $[\Lambda,n,0,\beta]$ be a skew semisimple stratum in $A$ with $V = \bigoplus_{i=1}^{\ell+1} V^i$ and $\beta = \Sigma_{i=1}^{\ell+1} \beta_i$ a splitting. We say that the stratum $[\Lambda,n,0,\beta]$ is {\it good}, if there exists a self-dual decomposition $V = \bigoplus_{j=-m}^m W^{(j)}$ which satisfies the following properties:
\begin{enumerate}
  \item $V = \bigoplus_{j=-m}^m W^{(j)}$ is {\it exactly subordinate} to $[\Lambda,n,0,\beta]$, in the sense of \cite[Definition 6.5]{St2}, that is, it is minimal among all self-dual decompositions which are properly subordinate to $[\Lambda,n,0,\beta],$\
  \item for $j \ne 0$, $W^{(j)}$ is contained in $V^{\ell+1}$,\
  \item for $j \ne 0$, $\dim_{E_{\ell+1}}W^{(j)}$ are all the same, say $f$.
\end{enumerate}
\end{defn}

We assume that $[\Lambda,n,0,\beta]$ is a good skew semisimple stratum in $A$. Then by definition we have
\begin{eqnarray}
\displaystyle W^{(0)} &=& \Bigl{(}\bigoplus_{i=1}^\ell V^i \Bigl{)} \oplus (W^{(0)} \cap V^{\ell+1}),\\
\displaystyle V^{\ell+1} &=& (W^{(0)} \cap V^{\ell+1}) \oplus \Bigl{(} \bigoplus_{j=-m, j \ne 0}^m W^{(j)} \Bigl{)},
\end{eqnarray}
where possibly $W^{(0)} \cap V^{\ell+1} = (0)$.
Set $E' = E_{\ell+1} = F[\beta_{\ell+1}]$ , and let $N$ be the positive integer defined by 
\begin{equation*}
\displaystyle \dim_F(\bigoplus_{j=-m, j \ne 0}^mW^{(j)}) = 2N.
\end{equation*}
 Since $\dim_{E'}(W^{(j)}) = f$, for $j \ne 0$, we have
\begin{equation}
N = m\dim_F(W^{(j)}) = m[E':F]f,
\end{equation}
where $[E':F]$ denotes the field extension degree of $E'/F$.

\section{Beta extensions}     

In this section, we assume that $[\Lambda,n,0,\beta]$ is a good skew semisimple stratum in $A$ with $V = \bigoplus_{i=1}^{\ell+1} V^i$ and $\beta = \Sigma_{i=1}^{\ell+1} \beta_i$ a splitting, as defined in the previous section. Set $E_i = F[\beta_i]$, for $1 \le i \le \ell+1$, and $E = \bigoplus_{i=1}^{\ell+1} E_i$.

In \cite{St1}, we have the $\mathfrak{o}_F$-lattices  $\widetilde{\mathfrak{H}}^t = \widetilde{\mathfrak{H}}^t(\beta,\Lambda),\ \widetilde{\mathfrak{J}}^t = \widetilde{\mathfrak{J}}^t(\beta,\Lambda)$ in $A = \mathrm{End}_F(V)$,
and the compact open subgroups $\widetilde{H}^t = \widetilde{H}^t(\beta,\Lambda),\ \widetilde{J}^t = \widetilde{J}^t(\beta,\Lambda)$ of $\widetilde{G}$, for $t = 0, 1$. Since these objects are $\sigma$-stable
, we obtain compact open subgroups of $G^+$ and $G$ as follows:
\begin{equation*}
J^+(\beta,\Lambda) = \widetilde{J}(\beta,\Lambda) \cap G^+,\ H^t(\beta,\Lambda) = \widetilde{H}^t(\beta,\Lambda) \cap G,\ J^t(\beta,\Lambda) = \widetilde{J}^t(\beta,\Lambda) \cap G,
\end{equation*}
for $t = 0, 1$. Then $J^+(\beta,\Lambda) = P^+(\Lambda_{\mathfrak{o}_E})J^1(\beta,\Lambda)$ and $J(\beta,\Lambda) = P(\Lambda_{\mathfrak{o}_E})J^1(\beta,\Lambda)$ (cf. \cite[3.1]{St2}). Put $J^\mathrm{o}(\beta,\Lambda) = P^\mathrm{o}(\Lambda_{\mathfrak{o}_E})J^1(\beta,\Lambda).$
Then we have $J^\mathrm{o}(\beta,\Lambda) \subset J(\beta,\Lambda) \subset J^+(\beta,\Lambda)$.

Denote by $\widetilde{\mathcal{C}}(\Lambda,0,\beta)$ the set of all semisimple characters of $\widetilde{H}^1(\beta,\Lambda)$, defined by \cite[Definition 3.13]{St1}. Put
\begin{equation*}
\mathcal{C}_-(\Lambda,0,\beta) = \{\theta \vert_{H^1(\beta,\Lambda)} \vert \theta \in \widetilde{\mathcal{C}}(\Lambda,0,\beta)\ \text{and}\ \theta^\sigma(x) = \theta(x)\ (x \in \widetilde{H}^1(\beta,\Lambda))\},
\end{equation*}
where $\theta^\sigma(x) = \theta(\sigma^{-1}(x))$, which is defined as in \cite[3.6]{St1}.

Let $B$ be the $A$-centralizer of $\beta = \sum_{i=1}^{\ell+1} \beta_i$. Then we have $B = \bigoplus_{i=1}^{\ell+1} B^i$, where $B^i = \mathrm{End}_{E_i}(V^i)$. As in \cite[4.2]{St2}, we choose a self-dual $\mathfrak{o}_F$-lattice sequence $\Lambda^\mathsf{M}$ in $V$ which satisfies (1) $\mathfrak{b}_0(\Lambda^\mathsf{M}) = \mathfrak{a}_0(\Lambda^\mathsf{M}) \cap B$ is a maximal self-dual $\mathfrak{o}_E$-order of $B = \bigoplus_{i=1}^{\ell+1} B^i$, and (2) $\mathfrak{b}_0(\Lambda^\mathsf{M}) \supset \mathfrak{b}_0(\Lambda)$. 
From \cite[Corollary 2.9]{St2}, there exists a self-dual $\mathfrak{o}_E$-lattice sequence $\Lambda^\mathsf{m}$ in $V$ which satisfies (1) $\mathfrak{b}_0(\Lambda^\mathsf{m})$ is a minimal self-dual $\mathfrak{o}_E$-order of $B$, and (2) $\mathfrak{a}_0(\Lambda^\mathsf{m}) \subset \mathfrak{a}_0(\Lambda)$.
Thus we have $\mathfrak{b}_0(\Lambda^\mathsf{m}) \subset \mathfrak{b}_0(\Lambda) \subset \mathfrak{b}_0(\Lambda^\mathsf{M}).$

Since the invariants $k_0(\beta,\Lambda^\mathsf{m})$ and $k_0(\beta,\Lambda^\mathsf{M})$ are negative integers, there exist integers $n_\mathsf{m}, n_\mathsf{M}$ such that $[\Lambda^\mathsf{m},n_\mathsf{m},0,\beta],\ [\Lambda^\mathsf{M},n_\mathsf{M},0,\beta]$ are skew semisimple strata in $A$.
 Let $\theta \in \mathcal{C}_-(\Lambda,0,\beta)$. Then from \cite[3.2]{St1}, there exist canonical bijections $\tau_{\Lambda,\Lambda^\mathsf{m},\beta}: \mathcal{C}_-(\Lambda,0,\beta) \simeq \mathcal{C}_-(\Lambda^\mathsf{m},0,\beta)$ and $\tau_{\Lambda,\Lambda^\mathsf{M},\beta}: \mathcal{C}_-(\Lambda,0,\beta) \simeq \mathcal{C}_-(\Lambda^\mathsf{M},0,\beta)$. Put $\theta_\mathsf{m} = \tau_{\Lambda,\Lambda^\mathsf{m},\beta}(\theta),\ \theta_\mathsf{M} = \tau_{\Lambda,\Lambda^\mathsf{M},\beta}(\theta).$

\begin{prop}$($\cite[Proposition 3.5]{St2}$)$         
There exists a unique irreducible representation $\eta$ (resp. $\eta_\mathsf{m}$, $\eta_\mathsf{M}$) of $J^1 = J^1(\beta,\Lambda)$ (resp. $J_\mathsf{m}^1 = J^1(\beta,\Lambda^\mathsf{m})$, $J_\mathsf{M}^1 = J^1(\beta,\Lambda^\mathsf{M})$) containing $\theta$ (resp. $\theta_\mathsf{m}$, $\theta_\mathsf{M}$).
\end{prop}

From the above choice of $\Lambda^\mathsf{m},\ \Lambda^\mathsf{M}$, we can form the group $J_{\mathsf{m},\mathsf{M}}^1 = P_1(\Lambda^\mathsf{m}_{\mathfrak{o}_E})J_\mathsf{M}^1.$

\begin{prop}$($\cite[Proposition 3.7 and Theorem 4.1]{St2}$)$    
\rm{(i)} There exists a unique irreducible representation $\eta_{\mathsf{m},\mathsf{M}}$ of $J_{\mathsf{m},\mathsf{M}}^1$ satisfying (1) $\eta_{\mathsf{m},\mathsf{M}} \vert J_\mathsf{M}^1 = \eta_\mathsf{M}$, (2) $\eta_{\mathsf{m},\mathsf{M}}$ and $\eta_\mathsf{m}$ induce equivalent irreducible representations of $P_1(\Lambda^\mathsf{m})$.

\rm{(ii)} There exists a representation $\kappa_\mathsf{M}$ of $J^+_\mathsf{M} = J^+(\beta,\Lambda^\mathsf{M})$ that extends $\eta_{\mathsf{m},\mathsf{M}}$.
\end{prop}

From \cite[Definition 4.5]{St2}, there exists an extension $\kappa$ of $\eta$ to $J^+$, which is called a {\it $\beta$-extension, relative to $\Lambda^\mathsf{M}$, and compatible with $\kappa_\mathsf{M}$}. The representation $\kappa$ depends only on $\Lambda^\mathsf{M}$, not on the choice of $\Lambda^\mathsf{m}$.

\section{Simple types}           

Let $[\Lambda,n,0,\beta]$ be a good skew semisimple stratum in $A$ with the splitting $V = \bigoplus_{i=1}^{\ell+1} V^i$ and $\beta = \sum_{i=1}^{\ell+1} \beta_i$, defined in section 2. Let $E_i = F[\beta_i]$, for $1 \le i \le \ell+1$, and $E = F[\beta] = \bigoplus_{i=1}^{\ell+1} E_i$. Set $\beta' = \beta_{\ell+1}$ and $E' = E_{\ell+1}$.

Let $\widetilde{M}$ be the stabilizer in $\widetilde{G}$ of $V = \bigoplus_{j=-m}^m W^{(j)}$. Then $\widetilde{M}$ is a $\sigma$-stable Levi subgroup of $\widetilde{G}$. Let $\widetilde{P}$ be a $\sigma$-stable parabolic subgroup with Levi factor $\widetilde{M}$, and $\widetilde{U}$ the unipotent radical of $\widetilde{P}$. We select the Lie algebra of $\widetilde{U}$ to be elements whose lower triangular block matrices are zero (cf. \cite[(7.1.13)]{BK1}). 
Then $P = \widetilde{P} \cap G,\ M^+ = \widetilde{M} \cap G^+,\ M = \widetilde{M} \cap G$ are parabolic subgroups of $G$ and Levi subgroups of $G^+$ and $G$, respectively. Let $U = \widetilde{U} \cap G$. Then $P = MU$ is a Levi decomposition. There exist isomorphisms
\begin{equation*}
\displaystyle M^+ \simeq G^{(0) +} \times \prod_{j=1}^m \widetilde{G}^{(j)},\ M \simeq G^{(0)} \times \prod_{j=1}^m \widetilde{G}^{(j)},
\end{equation*}
where $G^{(0) +}$ is the unitary group of the Hermitian space $(W^{(0)},h\vert_{W^{(0)}})$, $G^{(0)}$ is the identity component of $G^{(0) +}$, and $\widetilde{G}^{(j)} = \mathrm{Aut}_F(W^{(j)})$, for $1 \le j \le m$. By (2.3), we have $\dim_F(W^{(j)}) = N/m$, for $j \ne 0$. Hence there exist isomorphisms
\begin{equation*}
\widetilde{G}^{(j)} \simeq GL(N/m,F),
\end{equation*}
for $1 \le j \le m$.

\begin{prop}          
The subgroups $H^1(\beta,\Lambda),\ J^1(\beta,\Lambda)$, $J^\mathrm{o}(\beta,\Lambda)$, and $J(\beta,\Lambda)$ of $G$ have Iwahori decompositions with respect to $(M,P)$; letting $\mathcal{G}$ be any of those groups, we have
\begin{equation*}
\mathcal{G} = (\mathcal{G} \cap U^-)(\mathcal{G} \cap M)(\mathcal{G} \cap U),
\end{equation*}
where $U^-$ denotes the opposite of $U$ relative to $M$.
\end{prop}
\begin{proof}
Since $V = \bigoplus_{j=-m}^m W^{(j)}$ is properly subordinated to $[\Lambda,n,0,\beta]$ by Definition 2.1 (cf. \cite[Definition 6.5]{St2}), this follows from \cite[Corollary 5.10]{St2}.
\end{proof}

\begin{prop}         
(i) There exists a canonical isomorphism
\begin{equation*}
\displaystyle H^1(\beta,\Lambda) \cap M \simeq H^1(\beta,\Lambda^{(0)}) \times \prod_{j=1}^m\widetilde{H}^1(\beta',\Lambda^{(j)}),
\end{equation*}
with corresponding expressions for $J^1(\beta,\Lambda),\ J(\beta,\Lambda)$. 

(ii) Let $\theta \in \mathcal{C}_-(\Lambda,0,\beta)$. Then, under the above isomorphism, we have
\begin{equation*}
\displaystyle \theta \vert H^1(\beta,\Lambda) \cap M \simeq \theta^{(0)} \otimes \bigotimes_{j=1}^m (\widetilde{\theta}^{(j)})^2,
\end{equation*}
where $\theta^{(0)} \in \mathcal{C}_-(\Lambda^{(0)},0,\beta)$, $\widetilde{\theta}^{(j)} \in \widetilde{\mathcal{C}}(\Lambda^{(j)},0,\beta')$, and $(\widetilde{\theta}^{(j)})^2 \in \widetilde{\mathcal{C}}(\Lambda^{(j)},0,2\beta')$, for $1 \le j \le m$.
\end{prop}
\begin{proof}
This is a direct consequence of \cite[Corollary 5.11]{St2}.
\end{proof}

From Proposition 4.1, we can form the group
\begin{equation*}
J_P^1 = H^1(\beta,\Lambda)(J^1(\beta,\Lambda) \cap P).
\end{equation*}
Let $\theta \in \mathcal{C}_-(\Lambda,0,\beta)$, and $\eta$ be the unique irreducible representation of $J^1(\beta,\Lambda)$ containing $\theta$.
Denote by $\eta_P$ the natural representation of $J_P^1$ on the subspace of $(J^1(\beta,\Lambda) \cap U)$-fixed vectors in the space of $\eta$.
Then, by \cite[Lemma 5.12]{St2}, we have $J_P^1 \cap M = J^1(\beta,\Lambda) \cap M$ and
\begin{equation*}
\displaystyle \eta_P \vert J^1(\beta,\Lambda) \cap M \simeq \eta^{(0)} \otimes \bigotimes_{j=1}^m \widetilde{\eta}^{(j)},
\end{equation*}
where $\eta^{(0)}$ is a unique irreducible representation of $J^1(\beta,\Lambda^{(0)})$ containing $\theta^{(0)}$, and $\widetilde{\eta}^{(j)}$ is a unique irreducible representation of $\widetilde{J}^1(\beta',\Lambda^{(j)}) = \widetilde{J}^1(2\beta',\Lambda^{(j)})$ containing $(\widetilde{\theta}^{(j)})^2$.

We define compact open subgroups $J_P$ and $J^+_P$ of $G$ and $G^+$ by
\begin{equation*}
J_P = H^1(\beta,\Lambda)(J(\beta,\Lambda) \cap P),\ J^+_P(\beta,\Lambda) = H^1(\beta,\Lambda)(J^+(\beta, \Lambda) \cap P^+),
\end{equation*}
respectively. Let $\kappa$ be a $\beta$-extension of $\eta$ to $J^+(\beta,\Lambda)$, and $\kappa_P$ the natural representation of $J_P^+$ on the space of $(J^+(\beta,\Lambda) \cap U) = (J^1(\beta,\Lambda) \cap U)$-fixed vectors in $\kappa$. We also denote by $\kappa_P$ the restriction of $\kappa_P$ to $J_P$. Then from \cite[Proposition 6.1]{St2}, $\kappa_P \vert J_P^1 = \eta_P$ and $\kappa_P$ is irreducible.

\begin{prop}      
The representation $\kappa_P$ of $J_P$ satisfies the following conditions:
\begin{enumerate}
  \item $\mathrm{Ind}_{J_P}^{J(\beta,\Lambda)} \kappa_P \simeq \kappa \vert J(\beta,\Lambda)$,
  \item there exist irreducible representations $\kappa^{(0)}$ of $J(\beta,\Lambda^{(0)})$ extending  $\eta^{(0)}$, and $\widetilde{\kappa}^{(j)}$ of $\widetilde{J}(\beta',\Lambda^{(j)})$ extending $\widetilde{\eta}^{(j)}$, for $1 \le j \le m$, $($cf. Proposition 6.5$)$ such that
  \begin{equation*}
  \displaystyle \kappa_P \vert J_P \cap M \simeq \kappa^{(0)} \otimes \bigotimes_{j=1}^m \widetilde{\kappa}^{(j)}.
  \end{equation*}
\end{enumerate}
\end{prop}
\begin{proof}
A proof of the proposition can be found below \cite[Proposition 5.13]{St2}.
\end{proof}

There exist natural isomorphisms
\begin{equation*}
J_P/J_P^1 \simeq J(\beta,\Lambda)/J^1(\beta,\Lambda) \simeq P(\Lambda_{\mathfrak{o}_E})/P_1(\Lambda_{\mathfrak{o}_E}),
\end{equation*}
and, moreover, this quotient is isomorphic to
\begin{equation*}
\displaystyle P(\Lambda^{(0)}_{\mathfrak{o}_E})/P_1(\Lambda^{(0)}_{\mathfrak{o}_E}) \times \prod_{j=1}^m \widetilde{P}(\Lambda_{\mathfrak{o}_{E'}}^{(j)})/\widetilde{P}_1(\Lambda^{(j)}_{\mathfrak{o}_{E'}}).
\end{equation*}
Put $\overline{G}^{(0)} = P(\Lambda^{(0)}_{\mathfrak{o}_E})/P_1(\Lambda^{(0)}_{\mathfrak{o}_E}).$
Then $\overline{G}^{(0)}$ is (the group of rational points of) a reductive algebraic group defined over $k_{F_0}$, and is not always connected. We denote by $\overline{G}^\mathrm{o}$ the identity component of $\overline{G}^{(0)}$. 
Then we have $\overline{G}^\mathrm{o} = P^\mathrm{o}(\Lambda^{(0)}_{\mathfrak{o}_E})/P_1(\Lambda^{(0)}_{\mathfrak{o}_E})$. 
From (2.2) and the fact that $V = \bigoplus_{j=-m}^m W^{(j)}$ is exactly subordinate to $[\Lambda,n,0,\beta]$, the quotient $\widetilde{P}(\Lambda_{\mathfrak{o}_{E'}}^{(j)})/\widetilde{P}_1(\Lambda_{\mathfrak{o}_{E'}}^{(j)})$ is isomorphic to $GL(f,k_{E'})$. Via these isomorphisms, we identify
\begin{equation}
J_P/J_P^1 = J(\beta,\Lambda)/J^1(\beta,\Lambda) = P(\Lambda_{\mathfrak{o}_E})/P_1(\Lambda_{\mathfrak{o}_E}) = \overline{G}^{(0)} \times GL(f,k_{E'})^m.
\end{equation}

Let $\tau$ be an irreducible smooth representation of $J(\beta,\Lambda)$ trivial on $J^1(\beta,\Lambda)$. Then $\tau$ is the inflation to $J(\beta,\Lambda)$ of the representation
\begin{equation*}
\displaystyle \overline{\tau}_0 \otimes \bigotimes_{j=1}^m \overline{\widetilde{\tau}}^{(j)},
\end{equation*}
of $J(\beta,\Lambda)/J^1(\beta,\Lambda)$, where $\overline{\tau}_0$ and $\overline{\widetilde{\tau}}^{(j)}$ are representations of $\overline{G}^{(0)}$ and $GL(f,k_{E'})$ which are isomorphic to $J(\beta,\Lambda^{(0)})/J^1(\beta,\Lambda^{(0)})$ and $\widetilde{J}(\beta',\Lambda^{(j)})/\widetilde{J}^1(\beta',\Lambda^{(j)})$,
respectively, for $1 \le j \le m$.

From a $\beta$-extension $\kappa$ as above and $\tau$, we define a smooth representation $\lambda$ of $J(\beta,\Lambda)$ by
\begin{equation*}
\lambda = \kappa \otimes \tau.
\end{equation*}
From (4.1), we can regard $\tau$ as an irreducible smooth representation of $J_P$ trivial on $J^1_P$.

\begin{prop} $($\cite[Lemma 6.1]{St2}$)$     
Let $\lambda_P$ be the natural representation of the group $J_P = H^1(\beta,\Lambda)(J(\beta,\Lambda \cap P)$ on the space of $(J^1(\beta,\Lambda) \cap U)$-fixed vectors in $\lambda$. Then
\begin{enumerate}
  \item $\lambda_P$ is irreducible and $\mathrm{Ind}_{J_P}^{J(\beta,\Lambda)} \lambda_P \simeq \lambda$,
  \item $\lambda_P \simeq \kappa_P \otimes \tau$,
  \item letting $\lambda^{(0)} = \kappa^{(0)} \otimes \tau_0$ and $\widetilde{\lambda}^{(j)} = \widetilde{\kappa}^{(j)} \otimes \widetilde{\tau}^{(j)}$, for $1 \le j \le m$, we have
  \begin{equation*}
  \displaystyle \lambda_P \vert J_P \cap M \simeq \lambda^{(0)} \otimes \bigotimes_{j=1}^m \widetilde{\lambda}^{(j)},
  \end{equation*}
  where $\tau_0$ and $\widetilde{\tau}^{(j)}$ are representations of $J(\beta,\Lambda^{(0)})$ and $\widetilde{J}(\beta',\Lambda^{(j)})$, which inflate $\overline{\tau}_0$ and $\overline{\widetilde{\tau}}^{(j)}$ above, respectively.
\end{enumerate}
\end{prop}

\begin{defn}     
A representation $\lambda_P = \kappa_P \otimes \tau$ of $J_P$ as above is called a {\it simple type} in $G$, if $\tau$ satisfies the following conditions:
\begin{enumerate}
  \item $\overline{\tau}_0$ is an irreducible representation of $\overline{G}^{(0)}$ containing an irreducible cuspidal representation of $\overline{G}^\mathrm{o}$,
  \item $\overline{\widetilde{\tau}}^{(j)}$ is an irreducible cuspidal representation of $GL(f,k_{E'})$, for $1 \le j \le m$,
  \item $\overline{\widetilde{\tau}}^{(1)} \simeq \cdots \simeq \overline{\widetilde{\tau}}^{(m)}$.
\end{enumerate}
\end{defn}

\section{Self-dual simple types}       

Let $[\Lambda,n,0,\beta]$ be a good skew semisimple stratum in $A$ with splitting $V = \bigoplus_{i=1}^{\ell+1} V_i,\ \beta = \sum_{i=1}^{\ell+1} \beta_i$, defined in section 2. Let $\beta' = \beta_{\ell+1}$, $E_i = F[\beta_i]$, for $1 \le i \le \ell$, $E' = E_{\ell+1} = F[\beta']$, and $E = \bigoplus_{i=1}^{\ell+1} E_i$. Let $B$ be the $A$-centralizer of $\beta$, and $G_E = B \cap G$, as before. Then we have
\begin{equation*}
\displaystyle G_E = \prod_{j=1}^{\ell+1} G_{E_i},
\end{equation*}
where $G_{E_i}$ is (the group of rational points of) the restriction of scalars to $F_0$ of the connected unitary group of $(V^i, f_i)$, for $1 \le i \le \ell$, which is defined in section 2 (cf. \cite[p.299]{St2}).

Let $V = \bigoplus_{j=-m}^m W^{(j)}$ be the self-dual decomposition of $V$ in Definition 2.1.
Let $j > 0$. We take an (ordered) $\mathfrak{o}_{E'}$-basis $\{v_{j,1},\cdots,v_{j,f}\}$ of the lattice $\Lambda^{(j)}(0) = \Lambda(0) \cap W^{(j)}$ in $W^{(j)}$ such that it {\it splits} the lattice sequence $\Lambda^{(j)} = \Lambda \cap W^{(j)}$ (see \cite[Defnition 2.3]{St2}), and denote it by $\mathcal{B}^{(j)}$. As in \cite[6.2]{St2}, we take an (ordered) $E'$-basis $\mathcal{B}^{(-j)} = \{v_{-j,1},\cdots,v_{-j,f}\}$ for $W^{(-j)}$ that satisfies $f_{\ell+1}(v_{-j,s},v_{j,t}) = \varpi_{E'}\delta_{s,t}$, where $\overline{\varpi}_{E'}$ is a uniformizer of $E'$ which satisfies $\overline{\varpi}_{E'} = (-1)^{e(E' \vert E_0') - 1}\varpi_{E'}$ and $\delta_{s,t}$ denotes the Kronecker delta. 
We also choose a self-dual $\mathfrak{o}_E$-basis of the lattice $\Lambda^{(0)}(0) = \Lambda(0) \cap W^{(0)}$ for $W^{(0)}$ that splits the lattice sequence $\Lambda^{(0)} = \Lambda \cap W^{(0)}$, and denote it by $\mathcal{B}^{(0)}$.

Following \cite[6.2]{St2} again, we define Weyl group elements of $G_E$: For $j, k \ne 0,\ -m \le j, k \le m$, define $I_{j,k} \in B' = \mathrm{End}_{E'}(V')$ by $I_{j,k}(v_{k,s}) = v_{j,s}\ (1 \le s \le f),\ I_{j,k}(v_{\ell,s}) = 0\ (\ell \ne k)$, where $V' = \bigoplus_{j=-m, j \ne 0}^m W^{(j)}$.
For $1 \le j, k \le m$, we let $s_{j,k}, s_j$, and $s_j^{\varpi}$ be the elements defined in \cite[p.333]{St2}.
Then these elements belong to $B'$. Moreover we have $s_{j,k}, s_j, s_j^{\varpi} \in G^+_E = B \cap G^+$. In particular, the elements $s_j$ and $s_j^{\varpi}$ exchange the blocks $\Vec{e}^{(j)}A\Vec{e}^{(j)}$ and $\Vec{e}^{(-j)}A\Vec{e}^{(-j)}$, where $\Vec{e}^{(j)}$ denotes the projection $V \to W^{(j)}$ with kernel $\bigoplus_{k \ne j}W^{(k)}$.

For $1 \le j \le m$, we define an involution $\sigma_j$ on $\widetilde{G}^{(j)} = \mathrm{Aut}_F(W^{(j)})$ by using $s_j$ as follows: Identifying $\widetilde{G}^{(j)} = \{(\overline{g}^{-1},g) \in \widetilde{G}^{(j)} \times \widetilde{G}^{(-j)}\}$, we set
\begin{equation*}
\sigma_j(g) = s_jg(s_j)^{-1}\ (g \in \widetilde{G}^{(j)}).
\end{equation*}

\begin{prop}          
Let $[\Lambda,n,0,\beta]$ be a good skew semisimple stratum in $A$, and $(J_P,\lambda_P)$ a simple type in $G$ associated to it with $\lambda_P = \kappa_P \otimes \tau$. Let 
\begin{center}
$\kappa_P \vert J_P \cap M =\kappa^{(0)} \otimes \bigotimes_{j=1}^m \widetilde{\kappa}^{(j)}$ and\ 
$\lambda_P \vert J_P \cap M = \lambda^{(0)} \otimes \bigotimes_{j=1}^m \widetilde{\lambda}^{(j)}$.
\end{center}
Then we have
\begin{enumerate}
 \item $\widetilde{J}(\beta',\Lambda^{(1)}) \simeq \cdots \simeq \widetilde{J}(\beta',\Lambda^{(m)})$ and $\widetilde{J}(\beta',\Lambda^{(j)})$ is $\sigma_j$-stable, for $1 \le j \le m$,
 \item $\widetilde{\kappa}^{(1)} \simeq \cdots \simeq \widetilde{\kappa}^{(m)}$ and $\widetilde{\kappa}^{(j)} \circ \sigma_j \simeq \widetilde{\kappa}^{(j)}$, for $1 \le j \le m$,
 \item $\kappa^{(0)}$ is a $\beta$-extension of $\eta^{(0)}$, and $\widetilde{\kappa}^{(j)}$ is a 2$\beta'$-extension of $\widetilde{\eta}^{(j)}$, for $1 \le j \le m$,
 \item $\widetilde{\lambda}^{(1)} \simeq \cdots \simeq \widetilde{\lambda}^{(m)}$.
\end{enumerate}
\end{prop}
\begin{proof}
By definition 2.1, parts (1), (2), and (3) follow directly from \cite[Lemma 6.9, Corollary 6.10, and Proposition 6.3]{St2}, respectively. For part (4), we have $\widetilde{\lambda}^{(j)} = \widetilde{\kappa}^{(j)} \otimes \widetilde{\tau}^{(j)}$ by Proposition 4.4. Thus part (2) and Definition 4.5 show part (4). The proof is complete.
\end{proof}

\begin{defn} $\mathrm{(Selfdual\ simple\ type)}$     
Let $(J_P,\lambda_P)$ be a simple type in $G$ attached to a good skew semisimple stratum $[\Lambda,n,0,\beta]$ in $A$ with $\lambda_P = \kappa_P \otimes \tau$. The simple type $(J_P,\lambda_P)$ in $G$ is called {\it self-dual} if the representation $\widetilde{\tau}^{(j)}$ in Proposition 4.4 satisfies $\widetilde{\tau}^{(j)} \circ \sigma_j \simeq \widetilde{\tau}^{(j)}$, for $1 \le j \le m$.
\end{defn}

\begin{prop}         
Let $(J_P,\lambda_P)$ be a simple type in $G$ attached to a good skew semisimple stratum $[\Lambda,n,0,\beta]$ with $\lambda_P = \kappa_P \otimes \tau$. Let $\lambda^{(0)}, \widetilde{\lambda}^{(j)}$ be as in Proposition 5.2 for $\lambda_P \vert J_P \cap M$. Then $\lambda^{(0)}$ and $\widetilde{\lambda}^{(j)}$ are maximal simple types in $G^{(0)}$ and in $\widetilde{G}^{(j)}$, for $1 \le j \le m$, respectively. Moreover, if $\lambda_P = \kappa_P \otimes \tau$ is self-dual, then $\widetilde{\lambda}^{(j)} \circ \sigma_j \simeq \widetilde{\lambda}^{(j)}$, for $1 \le j \le m$.
\end{prop}
\begin{proof}
Since $V = \bigoplus_{j=-m}^m W^{(j)}$ is exactly subordinate to $[\Lambda,n,0,\beta]$ in Definition 2.1, it follows from \cite[Proposition 6.3 and Definition 6.17]{St2} that $\lambda^{(0)}, \widetilde{\lambda}^{(j)}$ are maximal simple types. 

From Proposition 5.2, for $\widetilde{\lambda}^{(j)} = \widetilde{\kappa}^{(j)} \otimes \widetilde{\tau}^{(j)}$, we have $\widetilde{\kappa}^{(j)} \circ \sigma_j \simeq \widetilde{\kappa}^{(j)}$, for $1 \le j \le m$. If $\lambda_P$ is self-dual, by Definition 5.2, we have
\begin{eqnarray*}
\widetilde{\lambda}^{(j)} \circ \sigma_j &\simeq& (\widetilde{\kappa}^{(j)} \otimes \widetilde{\tau}^{(j)}) \circ \sigma_j\\
                                         &\simeq& (\widetilde{\kappa}^{(j)} \circ \sigma_j) \otimes (\widetilde{\tau}^{(j)} \circ \sigma_j)\\
                                         &\simeq& \widetilde{\kappa}^{(j)} \otimes \widetilde{\tau}^{(j)} = \widetilde{\lambda}^{(j)},
\end{eqnarray*}
for $1 \le j \le m$. The proof is complete.
\end{proof}

\section{$G$-covers}       

Let $[\Lambda,n,0,\beta]$ be a good skew semisimple stratum in $A$ with splitting $V = \bigoplus_{i=1}^{\ell+1} V_i,\ \beta = \sum_{i=1}^{\ell+1} \beta_i$, defined in section 4, and $(J_P,\lambda_P)$ a simple type in $G$ attached to $[\Lambda,n,0,\beta]$. Let $E_i = F[\beta_i]$, for $1 \le i \le \ell+1$, $E = \bigoplus_{i=1}^{\ell+1} E_i$, and $G_E$ be the $G$-centralizer of $\beta$.
We have $G_E = \prod_{i=1}^{\ell+1}G_{E_i}$ as in section 5. 

As in \cite[6.3]{St2}, let $T_{E_i}$ be the maximal split torus of $G_{E_i}$ which corresponds to the basis $\mathcal{B}^{(0)} \cap V^i$, for $1 \le i \le \ell$, and $T_{E_{\ell+1}}$ the one of $G_{E_{\ell+1}}$ corresponding to the basis $\bigl{(}\bigcup_{j=-m, j \ne 0}^m \mathcal{B}^{(j)}\bigl{)} \cup \bigl{(}\mathcal{B}^{(0)} \cap V^{\ell+1}\bigl{)}$ (see (2.1) and (2.2)).
Put $T_E = \prod_{i=1}^{\ell+1} T_{E_i}$, and let $N$ be the normalizer of $T_E$ in $G_E$. Put $N_{\Lambda} = \{w \in N \vert \text{$w$ normalizes $P^\mathrm{o}(\Lambda_{\mathfrak{o}_E}) \cap M$}\}$, as in \cite[6.3]{St2}. For the elements $s_j$, and $s_j^{\varpi}$, defined in section 5, put $\zeta_j = \varepsilon s_js_j^{\varpi}$, for $1 \le j \le m$. We may arrange the order of the basis elements in $\bigcup_{j=-m}^m\mathcal{B}^{(j)}$ so that the element $\zeta_j$ of $T_E$ has a diagonal block form:
\begin{equation*}
\displaystyle \zeta_j = \mathrm{Diag}(\underbrace{1_f,\cdots,1_f,\varpi_{E'}1_f}_{m+1-j},1_f,\cdots,1_f,1_{W^{(0)}},1_f,\cdots,1_f,\underbrace{\overline{\varpi}_{E'}^{-1}1_f,1_f,\cdots,1_f}_{m+1-j}),
\end{equation*}
where $1_f$ and $1_{W^{(0)}}$ denotes the identity matrix in $\mathrm{End}_{E'}(W^{(j)})$ and $\mathrm{End}_E(W^{(0)})$ respectively. Denote by $\Vec{D}_{\Lambda}$ the abelian subgroup of $N_{\Lambda}$ generated by $\zeta_j$, for $1 \le j \le m$. Then $\Vec{D}_{\Lambda}$ consists of elements
\begin{equation*}
\displaystyle \prod_{j=1}^m \zeta_j^{n_{m+1-j}} = \mathrm{Diag}(\varpi_{E'}^{n_1}1_f,\cdots,\varpi_{E'}^{n_m}1_f,1_{W^{(0)}},\overline{\varpi}_{E'}^{-n_m}1_f,\cdots,\overline{\varpi}_{E'}^{-n_1}1_f),
\end{equation*}
for $(n_1,\cdots,n_m) \in \mathbb{Z}^m$. Thus there exists an isomorphism $\mathbb{Z}^m \simeq \Vec{D}_{\Lambda}$. 

For the simple type $\lambda_P = \kappa_P \otimes \tau$ in $G$, let $\rho$ be the irreducible cuspidal component of $\tau \vert P^\mathrm{o}(\Lambda_{\mathfrak{o}_E})$ (cf. Definition 7.3), and set
\begin{equation*}
N_{\Lambda}(\rho) = \{w \in N_{\Lambda} \vert ^w\rho \simeq \rho \}.
\end{equation*}
Then clearly $\Vec{D}_{\Lambda} \subset N_{\Lambda}(\rho)$. Denote by $I_G(\lambda_P)$ the space of $G$-intertwiners of $\lambda_P$, that is, $I_G(\lambda_P) = \{g \in G \vert I_g(\lambda_P) \ne (0)\}$.

\begin{prop}        
Let $(J_P,\lambda_P)$ be a simple type in $G$ attached to a good skew semisimple stratum $[\Lambda,n,0,\beta]$ in $A$. Then we have $I_G(\lambda_P) \subset J_PN_{\Lambda}(\rho)J_P$.
\end{prop}
\begin{proof}
Suppose that $g$ intertwines $\lambda_P = \kappa_P \otimes \tau$. Then we may assume that $g \in G_E$, since $I_G(\eta_P \vert {J_P^1}) = J^\mathrm{o}_PG_EJ^\mathrm{o}_P$.
 In a similar way to the proof of \cite[Lemma 5.12]{St2}, by Clifford Theory, the restriction of $\lambda_P$ to $J_P^\mathrm{o}$ has the form
\begin{equation*}
\displaystyle \lambda_P \vert J^\mathrm{o}_P = k\sum_p \{(\kappa_P \vert J^\mathrm{o}_P) \otimes \rho \}^p,
\end{equation*}
where $k$ is the multiplicity and the sum is taken over a set of representatives $P(\Lambda_{\mathfrak{o}_E})/N_{P(\Lambda_{\mathfrak{o}_E})}(\tau)$.
Since $g$ intertwines $\lambda_P \vert J^\mathrm{o}_P$, there exist $p_1, p_2 \in P(\Lambda_{\mathfrak{o}_E})$ such that $p_1gp_2$ intertwines $(\kappa_P \vert J^\mathrm{o}_P) \otimes \rho$. Thus, since $P(\Lambda_{\mathfrak{o}_E}) \subset J_P$, we may assume that $g$ intertwines $(\kappa_P \vert J^\mathrm{o}_P) \otimes \rho$. 
Hence, from \cite[Corollary 6.16]{St2}, $g \in J^\mathrm{o}_PN_{\Lambda}(\rho)J^\mathrm{o}_P$, whence $g \in J_PN_{\Lambda}(\rho)J_P$. This completes the proof.
\end{proof}

Let $\pi$ be a smooth representation of $GL(N/m,F)$. Let $\pi^\vee$ be the contragradient representation of $\pi$, and define the representation $\pi^*$ of $GL(N/m,F)$ by
\begin{equation*}
\pi^*(g) = \pi^\vee(\overline{g})\ (g \in GL(N/m,F)).
\end{equation*}
A smooth representation $\pi$ of $GL(N/m,F)$ is called $F/F_0$-{\it selfdual}, if it satisfies $\pi^* \simeq \pi$, (cf. \cite{M, MT}).

Identifying the Levi subgroup $M$ with $G^{(0)} \times GL(N/m,F)^m$, set
\begin{equation*}
J_M = J_P \cap M,\ \lambda_M = \lambda_P \vert J_M.
\end{equation*}

\begin{prop}      
Let $(J_P,\lambda_P)$ be a simple type in $G$ attached to a good skew semisimple stratum $[\Lambda,n,0,\beta]$ in $A$ . Then there exist irreducible supercuspidal representations $\pi_{\rm{cusp}}$ of $G^{(0)}$ and $\widetilde{\pi}^{(1)},\cdots,\widetilde{\pi}^{(m)}$ of $GL(N/m,F)$ such that the representations $\widetilde{\pi}^{(1)},\cdots,\widetilde{\pi}^{(m)}$ form a single inertial equivalence class, and such that $(J_M,\lambda_M)$ is a $[\pi_M,M]_M$-type in $M$, where $\pi_M = \pi_{\rm{cusp}} \otimes \bigotimes_{j=1}^m \widetilde{\pi}^{(j)}$.
If $(J_P,\lambda_P)$ is self-dual, the representations $\widetilde{\pi}^{(1)},\cdots,\widetilde{\pi}^{(m)}$ are inertially equivalent to a single irreducible $F/F_0$-selfdual supercuspidal representation.
\end{prop}
\begin{proof}
From Proposition 5.3, $\lambda_M = \lambda^{(0)} \otimes \bigotimes_{j=1}^m \widetilde{\lambda}^{(j)}$, and $\lambda^{(0)}$ and the $\widetilde{\lambda}^{(j)}$'s are all maximal simple types. 
From \cite[Theorem 7.14]{St2} and \cite[(6.2.3)]{BK1}, there exist an irreducible supercuspidal representation $\pi_{\rm{cusp}}$ of $G^{(0)}$ containing $\lambda^{(0)}$ and an irreducible supercuspidal representation $\widetilde{\pi}^{(j)}$ of $GL(N/m,F)$ containing $\widetilde{\lambda}^{(j)}$, for $1 \le j \le m$.
Since $\widetilde{\lambda}^{(1)} \simeq \cdots \simeq \widetilde{\lambda}^{(m)}$, then, again from \cite[(6.2.3)]{BK1}, $\widetilde{\pi}^{(1)},\cdots,\widetilde{\pi}^{(m)}$ are mutually inertially equivalent.

Suppose that $(J_P,\lambda_P)$ is self-dual. Fix $j$, $1 \le j \le m$. Then, from Proposition 5.3, we have $\widetilde{\lambda}^{(j)} \circ \sigma_j \simeq \widetilde{\lambda}^{(j)}$.
Thus $\widetilde{\pi}^{(j)} \circ \sigma_j$ contains $\widetilde{\lambda}^{(j)}$, and so there exists an unramified character $\chi$ of $GL(N/m,F)$ such that $\widetilde{\pi}^{(j)} \circ \sigma_j \simeq \widetilde{\pi}^{(j)}\chi$. Define a representation $\pi'$ of $GL(N/m,F)$ by $\pi' = \widetilde{\pi}^{(j)}\chi^{1/2}$.
Then we have
\begin{eqnarray*}
\pi' \circ \sigma_j = (\widetilde{\pi}^{(j)}\chi^{1/2}) \circ \sigma_j &=& (\widetilde{\pi}^{(j)} \circ \sigma_j)(\chi^{1/2} \circ \sigma_j)\\ 
                                                      &\simeq& (\widetilde{\pi}^{(j)}\chi)\chi^{-1/2} = \pi'.
\end{eqnarray*}
Employing a theorem of Gelfand and Kazhdan \cite[Theorem 2]{GK}, we have, for $g \in G$,
\begin{equation*}
\pi' \circ \sigma_j(g) = \pi'(^t\overline{g}^{-1}) \simeq (\pi')^\vee(\overline{g}) = (\pi')^*(g).
\end{equation*}
Hence $(\pi')^* \simeq \pi' \circ \sigma_j \simeq \pi'$, that is, $\pi' = \widetilde{\pi}^{(j)}\chi$ is inertially equivalent to $\widetilde{\pi}^{(j)}$, and is $F/F_0$-selfdual. This holds for any $j$, $1 \le j \le m$, whence the proof is completed.
\end{proof}

For the parabolic subgroup $P = MU$, let $P^-$ be the opposite parabolic subgroup of $P$, and $U^-$ the unipotent radical of $P^-$ with $P^- = MU^-$. The pair $(J_M,\lambda_M)$ satisfies the following conditions:
\begin{enumerate}
  \item $(J_P,\lambda_P)$ is a decomposed pair with respect to $(M,P)$, that is, 
  \begin{equation*}
  J_P = (J_P \cap U^-)J_M(J_P \cap U),
  \end{equation*}
  by Proposition 4.1, and $\lambda_P$ is trivial on $J_P \cap U$ and $J_P \cap U^-$,
  
  \item $\lambda_M = \lambda_P \vert J_P \cap M$.
\end{enumerate}

\begin{thm}        
Let $(J_P,\lambda_P)$ be a simple type in $G$ attached to a good skew semisimple stratum $[\Lambda,n,0,\beta]$ in $A$, and $\pi_M$ the irreducible supercuspidal representation of $M$ corresponding to $(J_M,\lambda_M)$ as in Proposition 6.2. Then $(J_P,\lambda_P)$ is a $G$-cover of $(J_M,\lambda_M)$, and so it is an $[M,\pi_M]_G$-type in $G$.
\end{thm}
\begin{proof}
By \cite[(8.3)]{BK2}, it is necessary to show that there exists an invertible element $\xi$ of $\mathcal{H}(G,\lambda_P)$ which is supported on $J_Pz_PJ_P$, where $z_P$ is an element of the center of $M$ and is strongly $(P,J_P)$-positive (cf. \cite[(6.16)]{BK2}). Define an element $z_P$ of the center of $M$ by
\begin{equation*}
\displaystyle z_P = \prod_{j=1}^m \zeta_j^{(m-j+1)e(E' \vert F)} \in \Vec{D}_{\Lambda}^-.
\end{equation*}
Then, from Proposition 6.3, there exists an element $\xi' \in \mathcal{H}(M,\lambda_M)$ supported on $J_Mz_P$, and thus $\xi = j_P(\xi') \in \mathcal{H}(G,\lambda_P)$ is the desired element. The proof is complete.
\end{proof}

\renewcommand{\labelenumi}{\theenumi}

 \vspace{4mm}

Kazutoshi Kariyama

 Department of Economics, Management

and Information Science,

Onomichi University,

Onomichi 722-8506, Japan

e-mail: kariyama@onomichi-u.ac.jp

\vspace{3mm}
Michitaka Miyauchi

Graduate school of Mathematics,

Kyoto University, Kitashirakawa,

Kyoto, 606-8502, Japan

e-mail: miyauchi@math.kyoto-u.ac.jp

\end{document}